\documentclass[10pt, english]{amsart}

\usepackage{amsmath,amssymb,enumerate}

\usepackage[T1]{fontenc}
\usepackage[all]{xy}

\usepackage{babel}
\usepackage{amstext}
\usepackage{amsmath}
\usepackage{amsfonts}
\usepackage{latexsym}
\usepackage{ifthen}

\usepackage{xypic}
\xyoption{all}
\newtheorem{lemma1}{}[section]

\newenvironment{lemma}{\begin{lemma1}{\bf Lemma.}}{\end{lemma1}}

\newenvironment{theorem}{\begin{lemma1}{\bf Theorem.}}{\end{lemma1}}
\newenvironment{proposition}{\begin{lemma1}{\bf Proposition.}}{\end{lemma1}}
\newenvironment{corollary}{\begin{lemma1}{\bf Corollary.}}{\end{lemma1}}
\newenvironment{remark}{\begin{lemma1}{\bf Remark.}\rm}{\end{lemma1}}

\newenvironment{notation}{\begin{lemma1}{\bf Notation.}}{\end{lemma1}}

\newenvironment{remark*}{{\bf Remark.}}{}
\newenvironment{example*}{{\bf Example.}}{}
\newenvironment{assumption*}{{\bf Assumption.}}{}

\makeatletter
\ifnum\@ptsize=0  \fi \ifnum\@ptsize=2  \fi 

\newcommand\sA{{\mathcal A}}

\newcommand\sI{{\mathcal I}}

\newcommand\sL{{\mathcal L}}

\newcommand\sO{{\mathcal O}}

\newcommand\sM{{\mathcal M}}

\title{The algebraic dimension of compact complex threefolds with vanishing second Betti number} 
\date{\today}

\author{Fr\'ed\'eric Campana}
\author{Jean-Pierre Demailly}
\author{Thomas Peternell}

\address{Fr\'ed\'eric Campana, Institut Elie Cartan,
Universit\'e Henri Poincar\'e,
BP 239,
F-54506. Vandoeuvre-les-Nancy C\'edex,
France\\
}
\email{frederic.campana@univ-lorraine.fr}

\address{Jean-Pierre Demailly, Universit\'ee de Grenoble-Alpes, Institut Fourier, UMR 5582 du C.N.R.S., 100 rue des Maths, 38610 Gi\`eres, France 
}
\email{ jean-pierre.demailly@univ-grenoble-alpes.fr
}

\address{Thomas Peternell, Mathematisches Institut, Universit\"at Bayreuth, 95440 Bayreuth, 
Germany}

\email{thomas.peternell@uni-bayreuth.de}

\begin{document}

\maketitle


\section{Introduction}

The paper \cite{CDP98} studied compact complex threefolds $X$ such that the second Betti number $b_2(X) = 0.$ The main result is based on Lemma 1.5, which happens to be incorrect in general (but might still hold in the context of the paper). 
In any case, some of the statements and proofs need to be adapted to fill the possible gaps; this is done in the present Corrigendum, with special regards 
 to potential complex structures on the $6$-sphere.

 \section{Statement of the results } 
 
We prove Theorem 2.1 in \cite{CDP98} in full generality in case $X$ has a meromorphic non-holomorphic map $X \dasharrow \mathbb P_1.$ In the remaining
case, $X$ has algebraic dimension $1$ and the algebraic reduction $f: X \to C$ is holomorphic.  In this case we prove that $c_3(X) \leq 0$; 
for simplicity, we will assume not only that $b_2(X) = 0$ but slightly stronger that $H^2(X,\mathbb Z) = 0$ and moreover that 
$H^1(X,\mathbb Z) = 0$, hence $C \simeq \mathbb P_1$. This suffices to treat the main application of 
complex structures on $S^6$. 

In summary, we shall prove 

\begin{theorem} \label{MT1}
Let $X$ be a $3-$dimensional compact complex manifold with $b_2(X) = 0$. Assume that there exists a non-holomorphic meromorphic non-constant map
$g: X \dasharrow \mathbb P_1$. 
Let $B$  be a holomorphic vector bundle on $X$. Then 
\begin{enumerate}
\item $H^{i}(X, B \otimes \sM) = 0 $ for $i \geq 0$ and $\sM \in {\rm Pic}^{\circ}(X) $ generic.
\item $\chi(X, B \otimes \sM) = 0 $ for all $\sM \in {\rm Pic}^{\circ}(X) $ 
\item $c_3(X) = 0$; i.e., either $b_1(X) = 0$ and $b_3(X) = 2$, or $b_1(X) = 1$ and \hbox{$b_3(X) = 0$}.
\end{enumerate}
\end{theorem}

Another proof has been given in \cite{LSS18}. 

Theorem \ref{MT1} takes care of all threefolds $X$ with $1 \leq a(X) \leq 2$ and $b_2(X) = 0$ except of those whose algebraic reduction $f: X \to C$ is holomorphic onto a curve $C$. 
In this case the general fiber has Kodaira dimension $\kappa (X_c) \leq 0$.

By topological considerations and surface classification, the general fiber of $f$ is either a torus, a primary Kodaira surface or a surface of type VII; in the latter case it is actually a primary Hopf surface or an Inoue surface, by
Lemma \ref{lemirrdcomp} and Lemma \ref{lem:fibers}. 
The case that the general fiber is a Kodaira surface is ruled out in Corollary \ref{cor:Kodaira}.
Then we show

\begin{theorem} \label{MT2}
Let $X$ be a $3-$dimensional compact complex manifold with $H^1(X,\mathbb Z) = H^2(X,\mathbb Z) = 0$ and algebraic dimension $a(X) = 1$. Assume that
the algebraic reduction $f: X \to C $ is holomorphic.
Then
\begin{enumerate}
\item $H^{i}(X, T_X \otimes \sM) = 0 $ for $i \ne1$.
\item $\chi(X, T_X \otimes \sM) \leq 0 $.
\item $c_3(X) \leq  0$. 
\end{enumerate}
\end{theorem}

As a consequence we deduce

\begin{corollary} Let $X$ be a 3-dimensional compact complex manifold homeomorphic to  $S^6$. Then $a(X) = 0$. 
\end{corollary} 

\begin{proof} Obviously, $a(X) \ne 3$, otherwise $X$ is Moishezon and therefore $b_2(X) \ne 0$. If $a(X) = 2,$ then there exists a meromorphic non-holomorphic map 
$g:X \dasharrow \mathbb P_1$. Then we conclude by Theorem \ref{MT1} that $c_3(X) = 0$. By Hopf's theorem, $c_3(X) = \chi_{\rm top}(S^6) = 2$, a contradiction. 
If $a(X) = 1 $ and the algebraic reduction $g: X \dasharrow C$ is not holomorphic, then $C \simeq \mathbb P_1$, and we conclude again by Theorem \ref{MT1}. 
If $a(X) = 1$ and  the algebraic reduction $g: X \dasharrow C$ is holomorphic, then we apply Theorem \ref{MT2} and obtain the same contradiction as before. 
\end{proof} 

We now comment on the strategy to prove Theorem \ref{MT2}. 
The arguments of \cite{CDP98} show the following 

\begin{proposition} \label{prop:CDP}
Let $X$ be a 3-dimensional compact complex manifold with $b_2(X) = 0$ and algebraic dimension $a(X) = 1$. Assume that the algebraic reduction $f: X\to C$ is 
holomorphic. 
If  \begin{equation}  \label{eq0a} R^2f_*(T_X \otimes \sL) = 0\end{equation} 
for some (or general) $\sL \in {\rm Pic}(X)$, then the assertions of Theorem \ref{MT2} hold. 

\end{proposition} 

Equation (\ref{eq0a}) is equivalent to the vanishing 
$$ H^2(X_c,T_X \otimes \sL \vert X_c) = 0 $$
for all complex-analytic fibers $X_c$ (with the natural fiber structure) of $f$, equivalently, 
$$ H^0(X_c,\Omega^1_X \otimes \sL^* \vert X_c) = 0.$$ 
The key is to show that the restriction $\sL \vert X_c$  of some or the general line bundle $\sL$ to any fiber is never torsion. 
Then we compute directly on $X_c$; here the case when $X_c$ is singular, in particular non-normal, needs special care.

For further informations on the problem of complex structures on $S^6$, we refer to \cite{Et15} and to volume 57 of the journal Differential Geometry and its Applications.

\section{Proof of Theorem \ref{MT1} }
Instead of simply pointing out the additions in the proof of \cite[Theorem 2.1]{CDP98} to be made, we
give full details for the benefit of the reader. 
Notice first that (\ref{MT1})(b) follows from  (\ref{MT1})(a) since $\chi(X,B \otimes \sM)$ does not depend on $\sM$, and 
 (\ref{MT1})(c) is a consequence of  (\ref{MT1})(b)
by applying Riemann-Roch to $B = T_X$ and $\sM = \sO_X$. So it remains to prove   (\ref{MT1})(a). By Serre duality,  (\ref{MT1})(a) needs only to be shown 
for $i = 0$ and $i = 2$. The case $i = 0$ follows from  \cite[Cor1.3]{CDP98}, since $X$ does carry effective non-zero divisors. Thus 
it remains to prove that
\begin{equation} \label{EQ1}  H^2(X,B \otimes \sM) \end{equation}
for generic, equivalently one, line bundle $\sM$. 
Let 
$g: X \dasharrow \mathbb P_1$
be a non-constant  non-holomorphic meromorphic map and 
$ \sigma:\hat X \to X$
be a resolution of indeterminacies of $g$. Let 
$$  f: \hat X \to C \simeq \mathbb P_1$$ 
be the fiber space given by the  Stein factorization of the holomorphic map $\sigma \circ g.$ 
Replacing $g$ by the induced meromorphic map $X \dasharrow C$, we may assume from the beginning that $g$ has connected fibers, 
hence no Stein factorization has to be taken. 
Note that the canonical map 
$$ H^2(X,B \otimes \sM) \to H^2(\hat X, \sigma^*(B \otimes \sM)) $$
is injective (by the Leray spectral sequence). 
Hence Equation (\ref{EQ1}) follows from 
\begin{equation} \label{EQ2} H^2(\hat X, \sigma^*(B \otimes \sM)) = 0. \end{equation} 
Fix an ample divisor $A$ on $C$. Then $f^*(\sO_C(A))$ can be written as
\begin{equation} \label{EQX} f^*(\sO_C(A)) = \sigma^*(\sL) \otimes \sO_{\hat X}(-E) \end{equation} 
with a line bundle $\sL$ on $X$ and a suitable effective divisor $E$ which is supported on the exceptional set of $\sigma$ and projects onto $C$. 
To verify Equation (\ref{EQ2}) it suffices to show that 
\begin{equation} \label{EQ3} H^2(\hat X, \sigma^*(B \otimes \sM) \otimes \sO_{\hat X}(-tE)) = 0 \end{equation} 
for some effective divisor $E$ supported on the exceptional locus of $\sigma$ and some $t \geq 0$. 
In fact, consider the exact sequence
$$ H^2(\hat X, \sigma^*(B \otimes \sM) \otimes \sO_{\hat X}(-tE)) \to H^2(\hat X, \sigma^*(B \otimes \sM)) \to H^2(tE, \sigma^*(B \otimes \sM)) $$
and note that 
$$ H^2(tE, \sigma^*(B \otimes \sM)) = 0. $$
This last vanishing is seen as follows: let $Z_t$ be the complex subspace of $X$ defined by the ideal sheaf $\sigma_*(\sO_{\hat X}(-tE))$. Then 
by 
$$ R^q(\sigma_{\vert tE})_*(\sO_{tE}) =  0$$
for $q = 1,2$ and the Leray spectral sequence, 
$$ H^2(tE, \sigma^*(B \otimes \sM)) = H^2(Z_t, B \otimes \sM). $$
Now the last group vanishes since $\dim Z_t = 1$.

By replacing $\sM$ by $\sM \otimes \sL^{t+k}$ for some positive integer $k$ and using (\ref{EQX}),   Equation (\ref{EQ3}) reads
\begin{equation} \label{EQ4} H^2(\hat X, \sigma^*(B \otimes \sM \otimes \sL^{t+k} \otimes \sO_{\hat X}(-tE))) = \end{equation} 
$$ =  H^2(\hat X, \sigma^*(B \otimes \sM \otimes \sL^k \otimes f^*(\sO_C(tA))). $$
By the Leray spectral sequence applied to $f$, Equation (\ref{EQ4}) comes down to verify 
\begin{equation} \label{EQ6} R^2f_*(\sigma^*(B \otimes \sM \otimes \sL^k)) = 0 \end{equation}
and 
\begin{equation} \label{EQ7} H^1(C, R^1f_*(\sigma^*(B \otimes \sM \otimes \sL^k)) \otimes \sO_C(tA)) =  0 \end{equation}
for suitable positive integers $k$ and $t$ and some line bundle $\sM$. 

To prove (\ref{EQ6}), let $C^* \subset C$ be the finite set of points $c \in C$ such that some component of the fiber $X_c$ does not meet $E$. In particular, 
if $X_c$ is irreducible, then $c \in C \setminus C^*$. Notice also that 
$$ R^2f_*(\sigma^*(B \otimes \sL^k)) \vert \{c\} \simeq H^2(\hat X_c,\sigma^*(B \otimes \sL^k)),$$
by the standard base change theorem. Applying Serre duality, we obtain
$$ H^2(\hat X_c,\sigma^*(B \otimes \sL^k)) \simeq H^0(\hat X_c,\sigma^*(B^* \otimes \sL^{-k}) \otimes \omega_{\hat X_c}) \simeq $$
$$ \simeq H^0(\hat X_c,\sigma^*(B^*) \otimes \sO_{\hat X_c}(-kE) \otimes \omega_{\hat X_c}). $$

We claim first that there is a number $k_0$ such that for $k \geq k_0$, 
\begin{equation} \label{eq:supp}  {\rm supp} \ R^2f_*(\sigma^*(B \otimes \sL^k)) \subset C^*.\end{equation} 
This is equivalent to saying that 
$$  H^0(\hat X_c, \sigma^*(B^* \otimes \sM^*) \otimes \sO_{\hat X_c}(-kE) \otimes \omega_{\hat X_c}) = 0$$ 
for $c \not \in C^*$. Fixing any point $c_0 \in C^*$, this number $k_0 = k_0(c)$ clearly exists; in case $X_c$ is reducible, we apply \cite[Prop.1.1]{CDP98}.
Hence the support of the direct image sheaf
$R^2f_*(\sigma^*(B \otimes \sL^{k_0}))$ is contained in a finite set $C_{k_0}$  in $C$. Since $\sigma^*(\sL) \vert X_c$ is effective, it follows that 
$C_{k} \subset C_{k_0}$. Thus, enlarging $k_0$ if necessary, (\ref{eq:supp}) is verified.

Hence we only need to consider the fibers $\hat X_c$ with $c \in C^*$. 
Let $P$ the set of line bundles $\sM$ of the form $$\sM = \sO_X(\sum m_i S_i)$$
with $m_i $ positive integers and $S_i$ fiber components of $f$ not meeting $E$ (the $S_i$ considered as surfaces in $X$). 

Since all line bundles $\sM \in P$ are trivial on $X \setminus f^{-1}(C^*)$, our previous considerations imply the existence of a number $k_0$ such that for all $k \geq k_0$ and all $\sM \in P$, 
$$ {\rm supp} \ R^2f_*(\sigma^*(B \otimes \sM \otimes \sL^k)) \subset C^*.$$ 
We are now going to construct a line bundle $\sM \in P$ such that 
$$ R^2f_*(\sigma^*(B \otimes \sM \otimes \sL^k)) = 0$$
for a suitable number $k$. 
Fix a point $c \in C^*$. Let $F_0 \subset X_c$ be the sub-divisor of $\hat X_c$ consisting of all components meeting $E$; let further $F_1 \subset \hat X_c$ be the sub-divisor consisting of all components which meet $F_1$ but not $E$. 
Continuing in this way we obtain a decomposition
$$ \hat X_c = \sum_{i=0}^rF_r $$
of sub-divisors $F_j \subset X_c$ who pairwise do not have common components and which have the property that all components of $F_j$ meet $F_{j-1}$ but do not meet $F_k$ for $k < j-1$. 
Now choose $m_r \gg 0 $ such that 
$$H^0(F_r, \sO_{\hat X}(-m_rF_{r-1}) \otimes \sigma^*(B^*) \otimes \omega_{\hat X_c} \vert F_r) = 0$$
This is possible by our construction. Next choose $m_{r-1} \geq m_r$ such that 
$$ H^0(F_{r-1}, \sO_{\hat X}(-m_{r-1}F_{r-2})  \otimes \sigma^*(B^*) \otimes \omega_{\hat X_c} \vert F_{r-1}) = 0.$$
Since ${\rm supp} (F_{r-2}) \cap {\rm supp}(F_r) = \emptyset$, we obtain
$$ H^0(F_{r-1} + F_r , \sO_{\hat X}(-m_{r-1}(F_{r-2} + F_{r-1} )  \otimes \sigma^*(B^*) \otimes \omega_{\hat X_c} \vert F_{r-1}) = 0.$$
Continuing in this way, we obtain a line bundle
$$ \sM'(c) = \sO_X\left(\sum_{i=2}^rm_iF_i\right)$$
(having in mind that the divisors $F_i, i \geq 1$ do not meet the exceptional locus of $\sigma$),
such that 
$$ H^0\left(\sum_{i=2}^r F_i, \sigma^*\big(B^* \otimes \sM'(c)\big)  \otimes \omega_{\hat X_c} \vert \Big(\sum_{i=2}^r F_i\Big) \right ) = 0.$$ 
Since the component $F_0$ meets $E$, it needs a special treatment. 
We observe that 
$$ \sO_{\hat X}(F_0) \simeq \sigma^* \big (\sO_X(\sigma(F_0) \big) \otimes \sO_{\hat X}(-E') $$
with some effective $\sigma$-exceptional divisor $E'$. Hence, choosing $m_1 \gg 0$ and setting 
$$ \sM(c) = \sM'(c) \otimes \sO_X(m_1\sigma(F_0)),$$
then 
$$ H^0 \left(\sum_{i=1}^r F_i, \sigma^*\big(B^* \otimes \sM(c)^*\big) \otimes \omega_{\hat X_c} \vert \Big(\sum_{i=1}^r F_i\Big) \right ) = 0.$$ 
Finally, enlarging $k$, we get 
$$ H^0(\hat X_c, \sigma^*(B^* \otimes \sM(c)^*) \otimes \sO_{\hat X}(-kE)\otimes  \omega_{\hat X_c} \vert X_c) = 0.$$ 
Setting
$$ \sM = \bigotimes_{c \in C^*} \sM(c),$$
this settles (\ref{EQ6}). \\
As to (\ref{EQ7}),  we fix $k$ as in (\ref{EQ6}) and then apply Serre's vanishing theorem to  the ample divisor $A$ to obtain $t$.

\section{General structure of the fibers}

From now on - for the rest of the paper -  we let $X$ be a compact complex manifold of dimension $3$ with 
$$H^1(X,\mathbb Z) = H^2(X,\mathbb Z) = 0$$ and holomorphic algebraic reduction $f:X \to C$ to the 
curve $C \simeq \mathbb P_1$.  We will freely use the theory of compact complex surfaces, in particular of non-K\"ahler surfaces, and refer to 
\cite{BHPV04} as general reference. Deviating from \cite{BHPV04}, we call bi-elliptic surfaces hyperelliptic.

An application of Riemann-Roch gives 

\begin{lemma} $\chi(D,\sO_D) = 0$ for all effective divisors $D$ on $X$. 
\end{lemma} 

\begin{proof} By Riemann-Roch,
$$ \chi(X,\sO_X(-D)) = \chi(X,\sO_X),$$
hence 
$$\chi(D,\sO_D) = \chi(X,\sO_X) - \chi(X,\sO_X(-D)) = 0.$$
\end{proof}

\begin{lemma} \label{lemirrdcomp} 
Let $s$ be the number of singular fibers and $r$ be the numbers of irreducible components of the singular fibers. Then 
$$ r = s - 1 + b_1(X_c),$$
where $X_c$ is a smooth fiber. 
Moreover $H_1(X_c, \mathbb Z) $ is torsion free for all smooth fibers~$X_c$. 
\end{lemma} 

\begin{proof} The first assertion is \cite[Lemma 3.2]{CDP98}. For the second, fix a smooth fiber $X_c$ and let $A \subset $ the union of all singular fibers of $f$ and set $X' = X \setminus A$. 
As seen in the proof of  \cite[Lemma 3.2]{CDP98}, 
$$ H_1(X',\mathbb Z) \simeq H_1(C \setminus f(A),\mathbb Z) \oplus H_1(X_c,\mathbb Z),$$
hence it suffices to show that $ H_1(X',\mathbb Z)$ is torsion free. 
To do this,  we consider the cohomology sequence for pairs, 
$$ 0 = H^4(X,\mathbb Z) \to H^4(A,\mathbb Z) \to H^5(X,A,\mathbb Z) \to H^5(X,\mathbb Z) \to 0.$$
Notice first that  $H^4(A,\mathbb Z)$ is torsion free. Further,
$H^5(X,\mathbb Z)$
is torsion free by the universal coefficent theorem, since $H_4(X,\mathbb Z)$ is torsion free: by Poincar\'e duality,
$$ H_4(X,\mathbb Z) \simeq H^2(X,\mathbb Z) = 0.$$  
Actually, $H^5(X,\mathbb Z) = 0$. 
Consequently, $$H^5(X,A,\mathbb Z) \simeq H_1(X',\mathbb Z)$$ is torsion free.
\end{proof}

\begin{lemma} \label{lem:fibers} 
Let $X_c$ be a smooth fiber of $f$. 
Then $X_c$ is either a  primary Hopf surface, an Inoue surface, a torus, a hyperelliptic surface with torsion free first homology group  or a primary Kodaira surface with torsion free first homology group. 

\end{lemma}

 \begin{proof} Note first that $K_{X_c} $ is topologically trivial, since $K_X$ is topologically trivial, due to $b_2(X) = 0$. 
Then we use the tangent sequence 
 $$ 0 \to T_{X_c} \to T_X \vert X_c \to N_{X_c/X} \simeq \sO_{X_c} \to 0 $$
 and observe that $c_2(X) = 0$, since $b_4(X) = 0$. Thus $c_2(X_c) = 0$. 
Since the (sufficiently) general fiber of an algebraic reduction has non-positive Kodaira dimension, \cite[Thm. 12.1]{Ue75},
so does every smooth fiber (see e.g. \cite[VI.8.1]{BHPV04}.
Then we conclude by surface classification
 and using the torsion freeness of $H_1(X_c,\mathbb Z) $, Lemma \ref{lemirrdcomp}.
 Note here that  $H_1(X_c,\mathbb Z)$ for a secondary Kodaira surface $X_c$ has torsion, since $c_1(X_c) $ is torsion in $H^2(X_c,\mathbb Z)$ (then apply the universal coefficient theorem)
 and that a secondary Hopf surface has torsion in $H_1(X_c, \mathbb Z)$ by definition.

 \end{proof}

\begin{corollary} 
\label{lem:irrfibers} 
All fibers
 of $f$ are irreducible unless the general fiber of $f$ is a torus, a hyperelliptic surface with torsion free first homology or a primary Kodaira surface with torsion free first homology. 
 \end{corollary}

We fix some notations for the rest of the paper. 

\begin{notation} \label{not} {\rm 
Let $ S \subset X$ be an irreducible reduced surface. In particular, $S$ is Gorenstein. We denote by $\omega_S$ the dualizing sheaf, which is a line bundle. 
Let 
$$ \eta: \tilde S \to S$$ be the normalization of $S$; denote by $N \subset S$ the non-normal locus, equipped with the
complex structure given by the conductor ideal. Let $\tilde N \subset \tilde S$ be the complex-analytic preimage. 
Let
$$ \pi: \hat S \to \tilde S$$
be a minimal desingularization and 
$$ \sigma: \hat S \to S_0$$
be a minimal model. 
For the class of $\omega_S$ we write $K_S$, analogously for $\omega_{\tilde S}$, etc. 

}
\end{notation} 

\begin{lemma} \label{lem:Mori} In the notations of (\ref{not}), we have
$$ \omega_{\tilde S} \simeq \eta^*(\omega_S) \otimes \sI_{\tilde N}$$
and 
$$ \omega_{\hat S} \simeq \pi^*\eta^*(\omega_S) \otimes \pi^*(\sI_{\tilde N}) \otimes \sO_{\hat S}(E) =  \pi^*\eta^*(\omega_S) \otimes \sO_{\hat S}(-\hat N) \otimes \sO_{\hat S}(- \hat E)$$
with an effective divisor $E$ supported on the exceptional locus of $\pi$ and $\hat N$ the strict transform of $\tilde N$ in $\hat S$.
Moreover, there are exact sequences
$$ 0 \to \sO_S \to \eta_*(\sO_{\tilde S}) \to \omega^{-1}_S \otimes \omega_N \to 0$$
and
$$ 0 \to \sO_N \to \eta_*(\sO_{\tilde N}) \to \omega^{-1}_S \otimes \omega_N \to 0$$
\end{lemma} 

\begin{proof} \cite[chap.3, sect.8]{Mo82}. 
\end{proof} 

As an immediate consequence, we note 

\begin{proposition} \label{prop:coho}  Let $S$ be any irreducible reduced compact Gorenstein surface with $\omega_S \equiv 0 $ and $\chi(S,\sO_S) = 0.$ Then 
\begin{enumerate}
\item $\chi(\tilde S,\sO_{\tilde S}) =  \chi(N,\omega^{-1}_S \otimes \omega_N) = - \chi(N,\sO_N)$;
\item $\chi(\tilde N,\sO_{\tilde N}) = 0.$
\end{enumerate} 
\end{proposition} 

\begin{proof} The first equation in (a) follows from Lemma \ref{lem:Mori}. As to the second equation in (a), observe by Serre duality
$$ \chi(N,\omega_S^{-1} \otimes \omega_N) = - \chi(N,{\omega_S}_{\vert N})  = - \chi(N,\sO_N),$$
since $\omega_S \equiv 0$. For the same reasons
$$ \chi(\tilde N,\sO_{\tilde N}) = \chi(N,\sO_N) + \chi(N,\omega_S^{-1} \otimes \omega_N) = 0.$$
\end{proof}

 \begin{proposition} \label{cor:VF}
 Let $X_c$ be a smooth fiber of $f$. Then
 $$ H^0(X_c,{T_X}_{ \vert X_c}) \ne 0,$$

 \end{proposition}
 
 \begin{proof}  Consider the exact sequence
 $$ 0 \to H^0(X_c,T_{X_c}) \to H^0(X_c,{T_X}_ {\vert X_c}) {\buildrel {\kappa} \over {\longrightarrow} } H^0(X_c, N_{X_c/X}) \simeq \mathbb C.$$ 
 If $H^0(X_c,T_{X_c}) \ne 0$, the assertion is clear. So it remains to treat the case that $X_c$ has no vector fields. Then by Lemma \ref{lem:fibers} and \cite{In74}, 
 $X_c$ is an Inoue surface of type $S_M$ or $S_N^{-}$, in which cases $H^1(X_c,T_{X_c}) = 0$. Thus $X_c$ is rigid and $\kappa $ is surjective, so that 
 $H^0(X_c,T_X \vert X_c) \ne 0$ also in these cases.
 
 \end{proof} 
 
 \begin{corollary} \label{cor:VF2} Let $X_c = \lambda S$ be a fiber with $S$ an irreducible singular surface and $\lambda \geq 1$.
 Then there exists a finite \'etale cover $S' \to S$ such that $H^0(S',T_{S'}) \ne 0$.
  \end{corollary} 
  
  \begin{proof} We consider the tangent sequence
  $$ 0 \to T_S \to {T_X}_{ \vert S}  {\buildrel \kappa \over {\longrightarrow}} N_{S/X}.$$ 
  If $\lambda = 1$, then $N_{S/X} \simeq \sO_S$. However, $\kappa $ is not surjective, since $S$ is singular.
  Hence 
  $$ H^0(S,T_S) = H^0(S,{T_X}_{ \vert S}).$$
  By semicontinuity and Proposition \ref{cor:VF}, $H^0(S,T_X \vert S) \ne 0$ and we conclude. \\
  If $\lambda \geq 2$, arguing in the same way, we obtain a torsion line bundle $\sL$ on $X$ such that 
  $$ H^0(S,T_S \otimes \sL_{\vert S}) \ne 0.$$
  Then we pass to a finite \'etale cover $\tilde S \to S$ to trivialize $\sL_{\vert S}$. 
  
  \end{proof} 
  
  \begin{remark} \label{rem:VF} A vector field $v \in H^0(S,T_S)$ induces canonically a vector field $v_0 \in H^0(S_0,T_{S_0})$. For brevity, we say that $v_0$ comes from $S$. 
  \end{remark}

\begin{proposition} \label{prop:normal}
Let $X_c = \lambda S$ with $S$ singular, irreducible. Then $S$ is non-normal. 
\end{proposition}  

\begin{proof} Suppose $S$ normal and let $\pi: \hat S \to S$ be a minimal desingularization and $\sigma: \hat S \to S_0$ a minimal model. \\
(a) Suppose that $S$ has only rational singularities, hence $S$ has only rational double points. 
Then $K_{\hat S} \equiv 0$, in particular $\hat S$ is a minimal surface containing $(-2)$-curves. 
By surface classification, $\hat S$ is either a K3 surface, an Enriques surface, of type VII or non-K\"ahler of Kodaira dimension $\kappa (\hat S) = 1$. The first two cases are impossible since
$$ \chi(\hat S,\sO_{\hat S}) = \chi(S,\sO_S) = 0.$$
If $\hat S$ is of type VII, then it must be a Hopf surface or an Inoue surface, since $K_{\hat S} \equiv 0$, but these surfaces do not contain 
$(-2)$-curves.  If $\kappa (\hat S) = 1$, then, since $\hat S$ has a vector field, it does not any rational curve, see e.g. \cite[Satz 1]{GH90}.

\medskip 
(b) Suppose now that $S$ has at least one irrational singularity. 
Then $\chi(\hat S,\sO_{\hat S}) < \chi(S,\sO_S) = 0, $
hence
$$ h^1(S_0,\sO_{S_0}) =  h^1(\hat S, \sO_{\hat S}) \geq 2.$$ 
Suppose first that $S_0$ is not K\"ahler. By classification, $S_0$ has to be a primary Kodaira
 surface or $\kappa (S_0) = 1$. By Corollary \ref{cor:VF2}, $H^0(S_0,T_{S_0}) \ne 0$ (up to finite \'etale cover).
Choose a non-zero vector field $v_0$ coming from $S$; then $v_0$ does not have zeroes by classification, \cite[Satz 1]{GH90}; note that in case $\kappa (S_0) = 1$, $S_0$ is an elliptic bundle over a curve of 
genus at least $2$. Hence we must have $\hat S = S_0$.
But then $\hat S$ does not contain contractible curves, so that $S$ is smooth, a contradiction. \\
Thus $S_0$ is K\"ahler. Since $K_{\hat S} = \pi^*(K_S) - E$ with $E$ a non-zero effective divisor,  $\kappa (S_0) = -\infty$ 
and $S_0$ is a ruled surface over a curve $B$ of genus $g = g(B) \geq 2$. 
Since $S$ has an irrational singularity, $\pi$ must contract an irrational curve whose normalization necessarily has genus at least $g$.
Thus $$h^0(S,R^1\pi_*(\sO_{\hat S})) \geq g$$
and therefore
$$1-g = \chi(\hat S,\sO_{\hat S}) \leq \chi(S,\sO_S) - g = - g,$$
which is absurd. 
\end{proof}

\section{Kodaira surfaces, hyperelliptic surfaces and tori} 

In this section we consider the case that the general fiber of $f$ is a Kodaira surface, a hyperelliptic surface or a torus. 
We rule out the case of Kodaira and hyperelliptic fibers and show in the torus case that for general line bundles $\sL$ on $X$, the restriction $\sL_{\vert X_c}$ to any fiber is never torsion.

\begin{proposition} \label{prop:localfreeness}  Assume that the general fiber of $f$ is a Kodaira surface or a torus.  Then
$R^jf_*(\sO_X)$ is locally free for all $j$, in fact, $h^j(X_c,\sO_{X_c})$ is independent on $c \in C$.
\end{proposition}

\begin{proof}  It suffices to show that $h^2(X_c,\sO_{X_c})$ is independent of $c$. Indeed, since $h^0(X_c,\sO_{X_c}) = 1 $ for all $c$ and since
$\chi(X_c,\sO_{X_c})$ is constant, $h^1(X_c,\sO_{X_c})$ does not depend on $c $ as well, and the assertions follow by Grauert's theorem. 
By Serre duality, 
$$H^2(X_c,\sO_{X_c}) = H^0(X_c,\omega_{X_c}) = H^0(X_c,\omega_X \vert X_c).$$
Setting $\sL = f_*(\omega_X)$, a locally free sheaf of rank one, we obtain
$$ \omega_X = f^*(\sL) \otimes \sO_X\Big(\sum_i (m_i -1) F_i\Big),$$
where $F_i$ are the non-reduced fiber components. In particular, ${\omega_X}_{ \vert X_c} = \sO_{X_c} $ for all reduced fibers $X_c$ and therefore 
$$ h^0(X_c,\omega_{X_c}) = 1 $$
for all those $c$. 
So let $X_c$ be a non-reduced fiber and set $Y = {\rm red}(X_c)$. 
We consider the complex subspace 
$$Z = \sum (m_i-1)  F_i $$ of $X_c$ and have an induced exact sequence
$$ 0 \to \sI_{Z/X_c} \otimes {\omega_X}_{\vert X_c} \simeq \sO_{X_c} \to {\omega_X}_{\vert X_c} \to {\omega_X}_{\vert Z} \to 0.$$ 
Applying $f_*$ and observing that 
$$ f_*(\sO_{X_c}) = f_*({\omega_X}_{\vert X_c} ) = \sO_{\{c\}},$$
shows that the restriction map
$$ H^0(X_c , {\omega_X}_{\vert X_c}) \to H^0(Z, {\omega_X}_{\vert Z}) $$
vanishes. 
Since 
$$ \dim  H^0(X_c , {\omega_X}_{\vert X_c}) = \dim H^0(X_c,\sO_{X_c}) = 1,$$
we conclude  $ h^0(X_c,\omega_{X_c}) = 1$.


\end{proof} 

\begin{corollary} \label{cor:restrict}  Assume that the general smooth fiber of $F$ is a Kodaira surface, a hyperelliptic surface or a torus. 
Then the restriction map
$$r_F: H^1(X,\sO_X) \to H^1(F,\sO_F) $$
is surjective. 
\end{corollary}

\begin{proof} Suppose first that $F$ is a Kodaira surface or a torus. Then the assertion is Theorem 3.1 in \cite{CDP98}; 
the proof works since we now know that $R^1f_*(\sO_X)$ is locally free. 
If $F$ is hyperelliptic, then $H^1(F,\sO_F)$ is one-dimensional, hence it suffices to show that $r_F \ne 0$. 
Let $\mu: H^1(X,\sO_X) \to {\rm Pic}(X)$ be the canonical isomorphism and write $ \mu(\alpha) = \omega_X$. Then $r_F(\alpha) \ne 0$. In fact,
otherwise $\omega_F = \omega_X \vert F = \sO_F$, noticing also that $c_1(\omega_F) = c_1(\omega_X \vert F) = 0$ since $H^2(X,\mathbb Z) = 0$. 
But $\omega_F \not \simeq \sO_F$, a contradition.
\end{proof} 

As a consequence, we obtain 

\begin{corollary} \label{cor:Kodaira}

The general fiber of $f$ cannot be a Kodaira or hyperelliptic surface.
\end{corollary}

\begin{proof} This is Proposition 3.6 in \cite{CDP98}. In the proof  of Proposition 3.6, Theorem 3.1 is used which is now established by Corollary \ref{cor:restrict}. 
Notice that in Step 2 of the proof of Proposition 3.6 in \cite{CDP98}, the local freeness of $R^jf_*(\tilde \sL)$ is used only generically. 
\end{proof}  

\begin{remark} The same arguments also rule out Hopf surfaces of algebraic dimension one. 
\end{remark}

From now - for the remainder of this section - we assume that the general fiber of $f$ is a torus.

\begin{proposition} \label{prop:triv}
$R^1f_*(\sO_X) ) \simeq  \sO_C(b_1)  \oplus \sO_C(b_2)$ with $b_j \geq 0$. 
\end{proposition} 

\begin{proof}  By Proposition \ref{prop:localfreeness}, the sheaf $R^1f_*(\sO_X)$ is locally free of rank two.
Write
$$ R^1f_*(\sO_X) = \sO_C(b_1) \oplus \sO_C(b_2).$$
We observe that  $R^1f_*(\sO_X)$ is generically spanned by
Corollary \ref{cor:restrict}, since 
$$ H^1(X,\sO_X) = H^0(C,R^1f_*(\sO_X)).$$ 
Hence $b_j \geq 0$. 

\end{proof} 

\begin{proposition} \label{prop:torsion1} 
For general  $\sL \in {\rm Pic}(X)$, the restriction $\sL_{\vert X_c}$ is non-torsion for all $c \in C$. 
\end{proposition} 
\begin{proof} 
By Proposition \ref{prop:triv}, $h^1(X,\sO_X) \geq 2 $ and the restriction 
$$ H^1(X,\sO_X) \to H^1(X_c,\sO_{X_c}) $$
is surjective for all $c$. Consequently, the kernel of the restriction 
$$ {\rm Pic}(X) = H^1(X,\sO_X) \to {\rm Pic}^{\circ}(X_c) = H^1(X_c,\sO_{X_c}) /H^1(X_c,\mathbb Z) $$
is discrete for all $c$ plus a linear subspace of codimension $2$.  Since $\dim C = 1$, 
it follows that for $\sL \in {\rm Pic}(X)$ general, the restriction $\sL_{ \vert X_c}$ is never trivial and thus also non-torsion. 
\end{proof}

\section{Hopf and Inoue surfaces} 

In this section we assume that the general fiber of $f$ is a Hopf or Inoue surface and show that for general line bundles $\sL$ on $X$, the restriction $\sL_{\vert X_c}$ is never torsion.

\begin{proposition} \label{prop:torsion2}
Assume that the general fiber of $f$ is a Hopf or Inoue surface. Let $\sL \in {\rm Pic}(X)$ be general. Then $\sL_{\vert X_c}$ is non-torsion for all $c \in C$, and the restriction map 
${\rm Pic}(X) \to {\rm Pic}(X_c)$ is surjective for any smooth fiber $X_c$. 
\end{proposition} 

\begin{proof} The exact sequence
$$ 0 \to \mathbb Z \to \mathbb C \to \mathbb C^* \to 1 $$
and our assumptions give 
$$ H^1(X,\mathbb C^*) = 0.$$ 
Moreover, $H^2(X,\mathbb C^*)$ is torsion. Consider the canonical morphism
$$\lambda: H^0(C,R^1f_*(\mathbb C^*)) \to H^2(C,\mathbb C^*) \simeq \mathbb C^*.$$
Then by the Leray spectral sequence, $\lambda $ is injective and the cokernel is torsion.
Hence
$$ H^0(C,R^1f_*(\mathbb C^*)) \simeq \mathbb C^*.$$
Choose
$$ 1 \ne u \in H^0(C,R^1f_*(\mathbb C^*)) $$
non-torsion. 
This section defines an inclusion
$$ \iota: \mathbb C^* \to R^1f_*(\mathbb C^*). $$
Let $C_0$ be the smooth locus of $f$ in $C$.
We claim that 
\begin{equation}\label{eq(*)} R^1f_*(\mathbb C^*) \vert C_0 = R^1f_{0*}(\mathbb C^*) \simeq \mathbb C^*. \end{equation} 
Suppose first that Claim (\ref{eq(*)}) holds. Then we conclude as follows. Certainly, $\iota $ is an isomorphism over $C_0$. Thus
we obtain a sequence
$$ 0 \to \mathbb C^* \to R^1f_*(\mathbb C^*) \to Q \to 0$$
where $Q$ is supported on the finite set $C \setminus C_0$. Since 
$H^0(C,R^1f_*(\mathbb C^*)) \simeq \mathbb C^*$ and since $H^1(C,\mathbb C^*) = 0$, it follows $H^0(C,Q) = 0$, hence $Q  = 0$. 
Thus $\iota$ is an isomorphism everywhere and consequently $u$ never takes value one, nor does - by our choice of $u$ - any multiple $u^m$. 
Hence $u$ defines a line bundle $\sL$ such that 
$\sL_{\vert X_c}$ is non-torsion for all $c \in C$. 
\vskip .2cm \noindent
It remains to prove Claim (\ref{eq(*)}). 
As before, set $\Delta = C \setminus C_0$, $A = f^{-1}(\Delta)$ and $X_0 = X \setminus A$.  Then, as in the proof of Lemma \ref{lemirrdcomp},
$$ H^4(A,\mathbb C^*) = H^1(X_0,\mathbb C^*) = H^1(C_0,\mathbb C^*) \oplus H^0(C_0,R^1f_{0*}(\mathbb C^*)).$$
Since $H^4(A,\mathbb C^*) \simeq (\mathbb C^*)^s$, it follows
$$ H^0(C_0,R^1f_{0*}(\mathbb C^*)) \simeq \mathbb C^*.$$ Since 
$R^1f_{0*}(\mathbb C^*)$ is locally constant of rank one, the claim follows. 

\end{proof} 

As a consequence we obtain

\begin{corollary} \label{cor:hopf-inoue} \begin{enumerate} 
\item $R^1f_*(\sO_X) \simeq \sO_C$;
\item $R^2f_*(\sO_X) = 0$;
\item for general $\sL \in {\rm Pic}(X)$ and all $c \in C$, we have $$H^0(X_c, \sL_{\vert X_c}) = H^0(X_c, \sL^*_{\vert X_c}) = 0.$$
\end{enumerate}
\end{corollary}

\begin{proof} (a) Since $R^1f_*(\sO_X)$ has rank one, we may write
$$ R^1f_*(\sO_X) \simeq \sO_C(a) \oplus {\rm torsion}.$$
By Proposition \ref{prop:torsion2}, $a = 0$. So it remains to show that $R^1f_*(\sO_X)$ is torsion free. If not, there exists a line bundle $\sM$,
such that $\sM_{\vert X_{c_0}}$ is non-torsion for some $c_0$ but $\sM_{\vert X_c} \simeq \sO_{X_c}$ for $c \ne c_0$. 
Write $ S = {\rm red}(X_{c_0})$. 
Then
$$\sM = f^*f_*(\sM) \otimes \sO_X(D),$$
with an effective divisor $D$ supported on $S$. Since $S$ is irreducible, $D = mS$ and therefore $\sO_X(D)_{\vert X_c}$ is a torsion line bundle,
contradiction. \\
(b) By (a), $h^1(X,\sO_X) = 1$. Since the general fiber of $f$ having negative Kodaira dimension, we have
 $$H^3(X,\sO_X) = H^0(X,\omega_X) = 0.$$
Thus we conclude from $\chi(X,\sO_X) = 0 $, that
$$H^2(X,\sO_X) = 0.$$
Hence, by the Leray spectral sequence, $R^2f_*(\sO_X)$ must be torsion free, therefore
$$ R^2f_*(\sO_X) = 0.$$ 
(c) As a consequence of (b), $R^2f_*(\sL) = 0$ for general $\sL$, hence 
$$ H^2(X_c,\sL_{X_c}) = 0$$
for all $c$. Thus 
$$ H^0(X_c,\sL_{X_c}) = 0$$
for general $\sL$ and all $c$ as well. 
In summary, we may say that 
$$ H^0(X_c,\sL_{\vert X_c}) = H^0(X_c,\sL^*_{\vert X_c})$$
for general $\sL$ and all $c$. \\
Now $\omega_X^{m} $ defines a section $t_m \in H^0(C,R^1f_*(\sO_X^*))$. Notice that for $c \in C_0$, the smooth locus of $f$, the bundle $\omega^m_{X_{\vert X_c}} = \omega^m_{\vert X_c})$ is never trivial and thus
$t_m$ does not take value $1$ on $C_0$. 
Since $R^1f_*(\sO_X^*) \simeq \sO_C^*$, the section never takes
value $1$, hence our claim follows. 
\end{proof} 

We will further need the following basic statement on Hopf and Inoue surfaces. 

\begin{proposition} 

\label{prop:easy}

Let $S$ be a primary Hopf surface. 
Assume that  $$ H^0(S,\Omega^1_S \otimes \sL) \ne 0 $$
for some line bundle $\sL$ on $S$. Then 
$$ H^0(S,\sL) \ne 0.$$
\end{proposition} 

\begin{proof} Choose a vector field $v$ on $S$ and let $C$ be the zero locus of $v$ which is purely one-dimensional. We obtain an
exact sequence
$$ 0 \to \sO_S(C) \to T_S \to \sO_S(-C) \otimes \omega^{-1}_S \to 0.$$
Dualizing, 
$$0 \to \sO_S(C) \otimes \omega_S  \to \Omega^1_S \otimes \sO_S(-C) \to 0.$$
Hence 
$$ H^0(S,\sO_S(C) \otimes \omega_S \otimes \sL) \ne 0$$
or 
$$H^0(S,\sO_S(-C) \otimes \sL) \ne 0,$$
In the latter case, the claim is clear. In the first case we observe that $$H^0(S,\omega^{-1}_S \otimes \sO_S(-C)) \ne 0.$$
Indeed, there exists another vector field $v'$, and 
$v \wedge v'$ is a section of $\omega^{-1}_S$ vanishing on $C$. 

\end{proof} 

\begin{proposition} \label{prop:easy1}
Let $S$ be an Inoue surface. Then there is a unique line bundle $\sL$ such that
$$ H^0(S,\Omega^1_S \otimes \sL) \ne 0.$$
Moreover, one of the following statements holds. 
\begin{enumerate}
\item either $H^0(S,T_S) \ne 0$ and $\sL = \omega_S^{-1}$
\item  or $H^0(S,T_S) = 0$ and $\sL^{\otimes 2} \simeq \omega^{-1}_S$. 
\end{enumerate} 

\end{proposition}

\begin{proof}  The existence of $\sL$ is classical, \cite{In74}. \\
If $S$ has a non-zero vector field $v$, necessarily without zeroes, then $v$ induces an exact sequence
$$ 0 \to \sO_S \to T_S \to \omega_S^* \to 0, $$
and the claim is immediate, since $S$ has no curves and since $H^0(S,\Omega^1_S) = 0$. 
If $H^0(S,T_S) = 0$, consider the exact sequence
$$ 0 \to \sL^* \to \Omega^1_S \to \sL \otimes \omega_S \to 0.$$ Since the sequence does not split, 
$$ H^1(S,\omega^{-1}_S \otimes  \sL^{\otimes -2}) \ne 0.$$
Hence either $\omega^{-1}_S \simeq \sO_{\otimes 2}$ or $ \omega_S \simeq \sL$, \cite[Lemma 1]{In74}. 
The second case however cannot happen, since $$H^0(S,\Omega^1_S \otimes \otimes \omega_S \simeq H^2(S,T_S) \ne 0,$$
\cite[Prop.2]{In74}. 

\end{proof}

\section{Proof of Theorem \ref{MT2} }

As already said in the introduction, it suffices to prove Proposition \ref{prop:CDP}. 
Thus we need to show that 
$$ H^2(X_c,{T_X}_{\vert X_c} \otimes \sL_{\vert X_c}) =  0$$
for all $c \in C$. 
By Serre duality, this comes down to show
that 
$$ H^0(X_c, {\Omega^1_X}_{\vert X_c} \otimes \sL_{\vert X_c}) = 0$$
for some $\sL \in {\rm Pic}(X) $ and for all $c \in C$. 

We first consider irreducible fibers. Let 
$$ S = {\rm red}X_c.$$
Using the (co-)tangent sheaf sequence
$$ 0 \to N^*_S \to {\Omega^1_X}_{\vert  S} \to  \Omega^1_{S} \to 0 $$
it is immediate that it suffices to show - provided $\sL_{\vert S}$ is non-torsion -  the following statement
\begin{equation} \label{eq:main1}  H^0(S,\tilde \Omega^1_S \otimes \sL_{\vert S}) = 0, \end{equation} 
where $$ \tilde \Omega^1_S = \Omega^1_S/{\rm torsion}.$$

We first treat smooth fibers $S = X_c$.

\begin{proposition}  Equations (\ref{eq:main1})  holds for smooth fibers $S$ (independent on the structure of the general fiber), i.e., for  $\sL \in {\rm Pic}(X) $ general 
$$  H^0(S, \Omega^1_S \otimes \sL_{\vert S}) =  0 $$
simultaneously for all smooth fibers $S$. 
\end{proposition} 

\begin{proof}
(a) First, if $S$ is a torus, then 
$$ H^0(S,T_S \otimes \sA)  = H^0(S,\Omega^1_S \otimes \sA) = 0$$
for all non-trivial $\sA$, hence we may take any $\sL$ such that $\sL_{\vert X_c}$ is never trivial, Proposition \ref{prop:torsion1}. \\
(b) If $S$ is a Hopf surface, then by Corollary  \ref{cor:hopf-inoue}, $H^0(S,\sL_{\vert S}) = 0$ for general $\sL$, hence we conclude by
Propositions \ref{prop:torsion2} and \ref{prop:easy}. \\
(c) If $S$ is an Inoue surface with a vector field, then for $\sL$ general, also $\sL^* \otimes \omega_X$ is general, hence 
$$ H^0(S,\sL^*_{\vert S} \otimes (\omega_X )_{\vert S}) = H^0(S,\sL^* \otimes \omega_S)) = 0,$$
hence we conclude by Propositions \ref{prop:torsion2} and \ref{prop:easy1}. \\
(d) Finally, assume that $S$ is an Inoue surface without vector field. Then we argue as in (c), observing that $(\sL^*)^{\otimes 2} \otimes \omega_X$ is general for
general $\sL$. 
\end{proof}

\begin{remark} \label{rem:torsion} 
Since the conormal bundle of a multiple fiber is torsion, the arguments also apply to fibers $X_c = \lambda S $ with $\lambda \geq 2$ and 
$S$ smooth. 
\end{remark} 

\begin{proposition} \label{prop:1}  Equation (\ref{eq:main1}) holds for singular reduced fibers.
\end{proposition} 

\begin{proof}  Recall the notations \ref{not}.  By Lemma \ref{lem:Mori} and Proposition \ref{prop:normal}, $\kappa (S_0 ) = - \infty$, the surface $S$ is non-normal and 
$$ H^2(\tilde S,\sO_{\tilde S}) = H^0(\tilde S,\omega_{\tilde S}) = 0.$$ 
Arguing by contradiction, there exists a one-dimensional family $\sL_t $ of line bundles on $X$ such that $$H^0(S, \tilde \Omega^1_S \otimes {\sL_t}_{\vert S}) \ne 0.$$
Passing to a desingularization and then to a minimal model $S_0$ as in Notation \ref{not}, 
there are numerically trivial line bundles $\sM_t$ on $S_0$ such that 
$$\sM_t = \sigma_* \pi^*\eta^*(\sL_t)^{**} $$
with a one-dimensional family of sections in 
$$ H^0(S_0,\Omega^1_{S_0} \otimes \sM_t).$$
Thus 
\begin{equation} \label{eq:non}  H^0(S_0,\Omega^1_{S_0} \otimes \sM_t) \ne 0. \end{equation} 
Observe that all line bundles  $\sM_t$ might be trivial. 
\vskip .2cm \noindent {\it Step 1.} 
Suppose first that $S_0$ is K\"ahler. Then by (\ref{eq:non}),  
$S_0$ must be ruled over a curve of genus at least two. 

{\it Claim.  $\tilde S$ has rational singularities, only.}

{\it Proof of the Claim.} Assume to the contrary that  $\tilde S$ has an irrational singularity. We claim that $H^1(\tilde S,\sO_{\tilde S}) = 0$. In fact, $\pi$ must contract a curve  $B_0$
projecting onto $B$. Thus $h^1(B_0,\sO_{B_0}) \geq g$ and therefore $h^0(\tilde S, R^1\pi_*(\sO_{\hat S})) \geq g$. Since $H^2(\tilde S,\omega_{\tilde S}) = 0$, the Leray spectral sequence
yields $H^1(\tilde S,\sO_{\tilde S}) = 0$ (and $h^0(\tilde S, R^1\pi_*(\sO_{\hat S})) = g$).

Thus all line bundles $\eta^*(\sL_t) $ are trivial
and we obtain a one-dimensional family $\tilde \omega_t$ of holomorphic one-forms on $\tilde S$. Moreover there exists a one-dimensional family $\omega_t$  of one-forms
on $B$ such that
$$ \sigma^* p^*(\omega_t) = \pi^*(\tilde \omega_t),$$
where $p: S_0 \to B$ is the ruling. 
Since $p(\sigma(B_0)) = B$, we have $\iota_{B_0}^*\sigma^* p^*(\omega_t)\ne 0$.
 On the other hand, since $\pi$ contracts $B_0$, it follows that $\iota_{B_0}^*(\pi^*(\omega_t)) = 0$, a contradiction.  This proves the {\it Claim} and thus $\tilde S$ has rational singularities, only.

In this case the morphism $p_0: S_0 \to B$ induced a morphism $\tilde p: \tilde S \to B$. 
In the language of divisors and using the notations of (\ref{not}) and (\ref{lem:Mori}) we have
$$ -K_{\hat S} \equiv \hat N + \hat E,$$
where $\hat N $ is the strict transform of $\tilde N$ in $\hat S$. 
Set 
$$ N_0 = \sigma_*(\hat N).$$
We are now using the theory of ruled surfaces as in \cite[section V.2]{Ha77}, taking over also the notations from \cite{Ha77}. 
In particular we have the invariant $e$ and a section $C_0$ with minimal self-intersection $C_0^2 = -e$. 
Moreover,
$$ -K_{S_0} \equiv 2C_0+ (e+2-2g)F,$$
where $F$ is a fiber of $p_0$ and $g = g(B) $ the genus of  $B$. 
Since $\tilde S$ has rational singularities, $\pi$ cannot contract any curve projecting onto $B$. Hence we must have
$$ N \equiv 2C_0 + aF$$ 
with $a \leq e+2-2g$.
Taking into account the numerical description of irreducible curves in $S_0$, as given in  \cite[section V.2]{Ha77}, it follows immediately that 
$e > 0 $ and that 
$$ N = 2C_0 +R$$
with an effective divisor $R \sim aF$ (note that the curve $C_0$ is the unique contractible curve in $S_0$). 
Consequently, $\tilde N$ has a unique component, say $\tilde N_1$, projecting onto $B$, and this component has multiplicity two. 
The map $\tilde p: \tilde S \to B$ induces a holomorphic map $p: S \to B'$ and a commutative diagram 
$$
\xymatrix{ 
\tilde S \ar[d]_{\tilde p} \ar[r]^{\eta} & S \ar[d]^{p}\\
B\ar[r]_{\tau} & B'.}
$$

The general fiber $S_b$ of $p$ is a reduced Gorenstein curve with 
$$ \omega_{S_b} \equiv 0$$ 
whose normalization of $S_b$ is a disjoint union of  smooth rational curves. Thus, if $S_b$ is irreducible, then $S_b$ is a rational curve with one
node or cusp, and if $S_b$ is reducible, it is a cycle of smooth rational curves. In case $S_b$ has a node or is a cycle, the normalization map 
$\eta$ is generically $2:1$ along $\tilde N_1$. In these cases however, $\tilde N_1$ would be reduced, see \cite{KW88}, a contradiction.  In the remaining case, 
$\eta$ has degree one along $\tilde N_1$, hence $\tau $ has degree one, too. Unless $\tau $ is an isomorphism and $g(B) = 2$, we have 
$h^1(B',\sO_{B'}) \geq 3$, hence $h^1(S,\sO_S) \geq 3$. Since $\chi(S,\sO_S) = 0$, we conclude
$$ h^0(S,\omega_S) = h^2(S,\sO_S) \geq 2.$$ 
Since $S$ is Moishezon and $\omega_S \equiv 0$, this is impossible. 
Alternatively, apply Proposition \ref{prop:triv}  or Corollary \ref{cor:hopf-inoue}, respectively.

Hence $\tau $ is biholomorphic, i.e., $p$ maps to the smooth curve $B$ of genus two and $h^1(S,\sO_S) = h^1(B,\sO_B) = 2$. 
Moreover, $h^0(S,\omega_S) = 1$, and therefore $\omega_S \simeq \sO_S$. 
The map $p$ being flat, $R^1p_*(\sO_S) $ is locally free of rank one, and by relative duality, 
$$ R^1p_*(\sO_S) \simeq p_*(\omega_{S/B})^* = \omega_B,$$
hence $$H^0(B,R^1p_*(\sO_S)) \ne 0.$$
But then $h^1(S,\sO_S) > h^1(B,\sO_B)$, a contradiction. This shows that $g(B) \geq 2$ is impossible and concludes the proof in the K\"ahler case.

\vskip .2cm \noindent  {\it Step 2.} 
We thus are reduced to the case that $S_0$ is not K\"ahler. \\
If $S_0$ is of type VII, then $H^0(S_0,\Omega^1_{S_0}) = 0,$ hence $H^0(S_0,\Omega^1_{S_0} \otimes \sM) = 0$ for $\sM$ general, contradicting (\ref{eq:non}).\\
The same argument applies to a secondary Kodaira surface $S_0$. If $S_0$ is a primary Kodaira surface, then the cotangent sequence reads 
$$ 0 \to \sO_{S_0} \to \Omega^1_{S_0} \to \sO_{S_0} \to 0,$$
which immediately gives a contradiction by tensorizing with $\sM_t$. \\
It remains to exclude the case $\kappa(S_0) = 1$. Since $H^0(S_0,T_{S_0}) \ne 0$, (\ref{cor:VF2}) and (\ref{rem:VF}),  the Iitaka fibration $h_0: S_0 \to B$ is an elliptic bundle over a curve of genus $g(B) \geq 2$, \cite[Satz 1]{GH90} and, as already noticed,  
the induced vector field $v_0$ has no zeroes. Hence $\tilde S = \hat S = S_0$. 
Since $\omega_{\tilde S} = \sI_{\tilde N} \otimes \eta^*(\omega_S)$, 
we have
$$ h^2(S,\sO_S) \geq  h^2(\tilde S, \sO_{\tilde S}),$$
hence $h^2(S,\sO_S) \geq 2$, contradicting Proposition \ref{prop:triv}  or Corollary \ref{cor:hopf-inoue}, respectively.

\end{proof} 

\begin{remark}  If the fiber $X_c = \lambda S$ with $S$ an irreducible reduced singular surface and $\lambda \geq 2$, we argue in the same way, passing to a finite \'etale cover. 
\end{remark}

Finally, we have to treat reducible fibers:

\begin{proposition}  Equation (\ref{eq:main1}) holds for reducible fibers.
\end{proposition} 

\begin{proof} Let $$F = \sum a_iS_i$$ be a reducible fiber.
Arguing by contradiction, there is a one-dimensional family $\sL_t$ of line bundles on $X$ such that 
$$ H^0(F,\Omega^1_X \vert F \otimes (\sL_t)_{\vert F}) \ne 0.$$ 
Hence there exists a number $i_0$ such that 
$$ H^0(S_{i_0}, \tilde \Omega^1_X \vert S_{i_0} \otimes ( \sL_t)_{\vert S_{i_0}}) \ne 0 $$
for all $t$, and therefore 
$$ H^0(S_{i_0}, \tilde \Omega^1_{S_{i_0}} \otimes ( \sL_t)_{\vert S_{i_0}}) \ne 0 $$
Now we argue as in Proposition \ref{prop:1} to obtain a contradiction. One might also use the line 
bundle $\sO_X(-kS_{i_1})$ for $k \gg 0$, where the surface $S_{i_1}$ meets in $S_{i_0}$ in a curve.

\end{proof}



\end{document}